\documentclass[11pt]{article}
\usepackage{amsmath}
\usepackage{amssymb}
\usepackage{amsthm}
\usepackage[usenames]{color}
\usepackage{amscd}
\usepackage{dsfont}
\usepackage{indentfirst}

\usepackage[colorlinks=true,linkcolor=blue,filecolor=red,
citecolor=webgreen]{hyperref}
\definecolor{webgreen}{rgb}{0,.5,0}

\numberwithin{equation}{section}

\def\N{{\mathds{N}}}
\def\Z{{\mathds{Z}}}

\def\1{{\bf 1}}

\def\id{\operatorname{id}}

\newtheorem{theorem}{Theorem}[section]

\newtheorem{lemma}[theorem]{Lemma}
\newtheorem{cor}[theorem]{Corollary}

\begin{document}

\title{{\bf Another generalization of Euler's arithmetic function and Menon's identity}}
\author{L\'aszl\'o T\'oth \\ \\ Department of Mathematics, University of P\'ecs \\
Ifj\'us\'ag \'utja 6, 7624 P\'ecs, Hungary \\ E-mail: {\tt ltoth@gamma.ttk.pte.hu}}
\date{}
\maketitle

\centerline{The Ramanujan Journal {\bf 57} (2022), 811--822}

\begin{abstract} We define the $k$-dimensional generalized Euler function $\varphi_k(n)$ as the number of ordered $k$-tuples
$(a_1,\ldots,a_k)\in \N^k$ such that $1\le a_1,\ldots,a_k\le n$ and both the product $a_1\cdots a_k$ and the sum $a_1+\cdots +a_k$
are prime to $n$. We investigate some of properties of the function $\varphi_k(n)$, and obtain a corresponding Menon-type identity.
\end{abstract}



\section{Motivation}

Jordan's arithmetic function $J_k(n)$ is defined as the number of ordered $k$-tuples $(a_1,\ldots,a_k)\in \N^k$ such that
$1\le a_1,\ldots,a_k\le n$ and the gcd $(a_1,\ldots,a_k,n)=1$. It is well-known that $J_k(n)$ is multiplicative in $n$ and
$J_k(n)=n^k \prod_{p\mid n} (1-1/p^k)$. If $k=1$, then $J_1(n)=\varphi(n)$ is Euler's arithmetic function.

A Menon-type identity concerning the function $J_k(n)$, obtained by Nageswara Rao \cite{Nag1972}, is given by
\begin{equation} \label{Nag_id}
\sum_{\substack{a_1,\ldots,a_k=1\\ (a_1,\ldots, a_k,n)=1}}^n
(a_1-1,\ldots,a_k-1,n)^k = J_k(n) \tau(n) \quad (n\in \N),
\end{equation}
where $\tau(n)$ is the number of divisors of $n$. If $k=1$, then \eqref{Nag_id} reduces to Menon's original identity \cite{Men1965}.

Euler's arithmetic function and Menon's identity have been generalized in various directions by several authors. See, e.g., the books \cite{McC1986,SC2004}, the papers \cite{JWActa,Nag1972,Sit1978,Tot2019,WJRam} and their references.

The function $X(n)= \# \{(a,b)\in \N^2: 1\le a,b\le n, (ab,n)=(a+b,n)=1 \}$ is an analog of Euler's $\varphi$-function, and was
introduced by Arai and Gakuen \cite{AG}. It was shown by Carlitz \cite{C} that the function $X(n)$ is multiplicative and
\begin{equation*}
X(n)= n^2 \prod_{p\mid n} \left(1-\frac1{p} \right) \left(1-\frac{2}{p} \right) \quad (n\in \N).
\end{equation*}

Note that if $n$ is even, then $X(n)=0$. The function $X(n)$ can also be given as
\begin{equation*}
X(n)= \varphi(n)^2 \sum_{d\mid n} \frac{\mu(d)}{\varphi(d)}  \quad (n\in \N),
\end{equation*}
$\mu$ denoting the M\"{o}bius function. The corresponding Menon-type identity
\begin{equation} \label{Sit_Menon}
\sum_{\substack{a,b=1 \\ (ab,n)=(a+b,n)=1}}^n (a+b-1,n) = X(n) \tau(n) \quad (n\in \N)
\end{equation}
was deduced by Sita Ramaiah \cite[Cor.\ 10.4]{Sit1978}. In fact, \eqref{Sit_Menon} is a corollary of a more general identity involving Narkiwicz-type
regular systems of divisors and $k$-reduced residue systems.

Recently, identity \eqref{Sit_Menon} was generalized by Ji and Wang \cite{JWActa,WJRam} to residually finite Dedekind domains, by using Narkiwicz-type
regular systems of divisors, and to the ring of algebraic integers, concerning Dirichlet characters modulo $n$, respectively. Note that in paper
\cite{WJRam} identity \eqref{Sit_Menon} is called the ``Arai-Carlitz identity''. However, Arai and Carlitz only considered the function $X(n)$ and did not
deduce such an identity. We refer to \eqref{Sit_Menon} as the Sita Ramaiah identity.

It is natural to introduce and to study the following $k$-dimensional generalization of the function $X(n)$, and to ask if the corresponding
generalization of the Sita Ramaiah identity is true for it. These were not investigated in the literature, as far as we know.
For $k\in \N$ we define the function $\varphi_k(n)$ as
\begin{equation} \label{def_varphi_k_n}
\varphi_k(n):= \sum_{\substack{a_1,\ldots, a_k=1 \\ (a_1\cdots a_k,n)=1\\ (a_1+\cdots+a_k,n)=1}}^n 1.
\end{equation}

Note that $\varphi_1(n)=\varphi(n)$ is Euler's function and $\varphi_2(n)=X(n)$ of above.
We investigate some of properties of the function $\varphi_k(n)$, and obtain a corresponding Menon-type identity.
Our main results are included in Section \ref{Section_Main_results}, and their proofs are presented in Sections
\ref{Section_Proofs_Euler} and \ref{Section_Proofs_Menon}.

We will use the following notations: $\id_k(n)=n^k$, $\1(n)=1$ ($n\in \N$), $\omega(n)$ will denote the number of distinct prime factors of
$n$, and ``$*$'' the Dirichlet convolution of arithmetic functions.

\section{Main results} \label{Section_Main_results}

In this paper we prove the following results.

\begin{theorem} \label{Theorem_Euler_gen_id} For every $k,n\in \N$,
\begin{align} \label{gen_1}
\varphi_k(n) = & \ \varphi(n)^k \prod_{p\mid n} \left(1- \frac1{p-1}+\frac1{(p-1)^2}-\cdots +(-1)^{k-1} \frac1{(p-1)^{k-1}}\right) \\
= & \ n^k \prod_{p\mid n} \left(1-\frac1{p} \right) \left(\left(1-\frac{1}{p}\right)^k - \frac{(-1)^{k}}{p^k} \right). \label{gen_2}
\end{align}
\end{theorem}

It is a consequence of Theorem \ref{Theorem_Euler_gen_id} that the function $\varphi_k(n)$ is multiplicative. Also, $\varphi_k(n)=0$ if and only
if $k$ and $n$ are both even. Further properties of $\varphi_k(n)$ can be deduced. Its average order is given by the next result.

\begin{theorem} \label{Theorem_asympt_varphi_k} Let $k\ge 2$ be fixed. Then
\begin{equation*}
\sum_{n\le x} \varphi_k(n)= \frac{C_k}{k+1}x^{k+1} +O\left(x^k (\log x)^{k+1}\right),
\end{equation*}
where
\begin{equation*}
C_k= \prod_p \left(1+\frac1{p^{k+1}}\left(\left(1-\frac1{p}\right)\left((p-1)^k-(-1)^k\right)-p^k \right) \right).
\end{equation*}
\end{theorem}

\begin{cor} \textup{($k=2$)} We have
\begin{equation*}
\sum_{n\le x} X(n)= \frac{C_2}{3}x^3 +O\left(x^2 (\log x)^3\right),
\end{equation*}
where
\begin{equation*}
C_2= \prod_p \left(1-\frac{3}{p^2}+ \frac{2}{p^3} \right) \approx 0.286747.
\end{equation*}
\end{cor}

We have the following generalization of Menon's identity.

\begin{theorem} \label{Theorem_Menon_gen} Let $f$ be an arbitrary arithmetic function. Then for every $k,n\in \N$,
\begin{equation} \label{Menon_gen_f}
\sum_{\substack{a_1,\ldots, a_k=1 \\ (a_1\cdots a_k,n)=1 \\ (a_1+\cdots+a_k,n)=1}}^n f((a_1+\cdots +a_k-1,n))
= \varphi_k(n)\sum_{d\mid n} \frac{(\mu \ast f)(d)}{\varphi(d)}.
\end{equation}
\end{theorem}

\begin{cor} \textup{($f(n)=n$)} For every $k,n\in \N$,
\begin{equation} \label{Menon_gen}
\sum_{\substack{a_1,\ldots, a_k=1 \\ (a_1\cdots a_k,n)=1 \\ (a_1+\cdots+a_k,n)=1}}^n (a_1+\cdots +a_k-1,n) = \varphi_k(n)\tau(n).
\end{equation}
\end{cor}

If $k=1$, then \eqref{Menon_gen} reduces to Menon's identity and if $k=2$, then it gives the Sita Ramaiah identity \eqref{Sit_Menon}.

\section{Proofs of Theorems \ref{Theorem_Euler_gen_id} and \ref{Theorem_asympt_varphi_k}} \label{Section_Proofs_Euler}

We need the following lemmas.

\begin{lemma} \label{Lemma_cong} Let $n,d\in \N$, $d\mid n$ and let $r\in \Z$. Then
\begin{equation*}
\sum_{\substack{a=1\\ (a,n)=1\\ a\equiv r \, \text{\rm (mod $d$)} }}^n 1 =
\begin{cases} \displaystyle \frac{\varphi(n)}{\varphi(d)},
& \text{ if $(r,d)=1$}, \\ 0, & \text{ otherwise}.
\end{cases}
\end{equation*}
\end{lemma}

Lemma \ref{Lemma_cong} is known in the literature, usually proved  by the inclusion-exclusion principle. See, e.g., \cite[Th.\ 5.32]{Apo1976}.
The following generalization and a different approach of proof was given in our paper \cite{Tot2019}.

\begin{lemma} {\rm (\cite[Lemma 2.1]{Tot2019}} \label{Lemma_cong_gen} Let $n,d,e\in \N$, $d\mid n$, $e\mid n$ and let $r,s\in \Z$. Then
\begin{equation*}
\sum_{\substack{a=1\\ (a,n)=1\\ a\equiv r \, \text{\rm (mod $d$)} \\ a\equiv s \, \text{\rm (mod $e$)} }}^n 1 =
\begin{cases} \displaystyle \frac{\varphi(n)}{\varphi(de)}(d,e),
& \text{ if $(r,d)=(s,e)=1$ and $(d,e) \mid r-s$}, \\ 0, & \text{ otherwise}.
\end{cases}
\end{equation*}
\end{lemma}

In the case $e=1$, Lemma \ref{Lemma_cong_gen} reduces to Lemma \ref{Lemma_cong}.

We need to define the following slightly more general function than $\varphi_k(n)$:
\begin{equation} \label{def_var_m_n}
\varphi_k(n,m):= \sum_{\substack{a_1,\ldots, a_k=1 \\ (a_1\cdots a_k,n)=1\\ (a_1+\cdots+a_k,m)=1}}^n 1.
\end{equation}

If $m=n$, then $\varphi_k(n,n)=\varphi_k(n)$, given by \eqref{def_varphi_k_n}.

\begin{lemma} {\rm (recursion formula for $\varphi_k(n,m)$)} \label{Lemma_recursion_varphi} Let $k\ge 2$ and $m\mid n$. Then
\begin{equation*}
\varphi_k(n,m)= \varphi(n) \sum_{d\mid m} \frac{\mu(d)}{\varphi(d)} \varphi_{k-1}(n,d).
\end{equation*}
\end{lemma}

\begin{proof}[Proof of Lemma {\rm \ref{Lemma_recursion_varphi}}]
We have
\begin{equation*}
\varphi_k(n,m)= \sum_{\substack{a_1,\ldots, a_k=1 \\ (a_1\cdots a_k,n)=1}}^n \sum_{d\mid (a_1+\cdots+a_k,m)} \mu(d)
\end{equation*}
\begin{equation*}
= \sum_{d\mid m} \mu(d) \sum_{\substack{a_1,\ldots, a_k=1 \\ (a_1\cdots a_k,n)=1\\ a_1+\cdots +a_k\equiv 0 \text{ (mod $d$) }}}^n 1
\end{equation*}
\begin{equation*}
= \sum_{d\mid m} \mu(d) \sum_{\substack{a_1,\ldots, a_{k-1}=1 \\ (a_1\cdots a_{k-1},n)=1}}^n \sum_{\substack{a_k=1\\ (a_k,n)=1\\ a_k\equiv -a_1-\cdots -a_{k-1} \text{ (mod $d$)}}}^n  1.
\end{equation*}

By using Lemma \ref{Lemma_cong} we deduce that
\begin{equation*}
\varphi_k(n,m) = \sum_{d\mid m} \mu(d) \sum_{\substack{a_1,\ldots, a_{k-1}=1 \\ (a_1\cdots a_{k-1},n)=1\\ (a_1+\cdots+a_{k-1},d)=1}}^n \frac{\varphi(n)}{\varphi(d)}
\end{equation*}
\begin{equation*}
= \varphi(n) \sum_{d\mid m} \frac{\mu(d)}{\varphi(d)} \sum_{\substack{a_1,\ldots, a_{k-1}=1 \\ (a_1\cdots a_{k-1},n)=1\\ (a_1+\cdots+a_{k-1},d)=1}}^n 1
\end{equation*}
\begin{equation*}
= \varphi(n) \sum_{d\mid m} \frac{\mu(d)}{\varphi(d)} \varphi_{k-1}(n,d),
\end{equation*}
where $d\mid m$ and $m\mid n$ imply that $d\mid n$.
\end{proof}

\begin{proof}[Proof of Theorem {\rm \ref{Theorem_Euler_gen_id}}]
Let $k,n,m\in \N$ such that $m\mid n$. We show that
\begin{equation} \label{form_varphi_k_n_m}
\varphi_k(n,m) =  \varphi(n)^k \prod_{p\mid m} \left(1- \frac1{p-1}+\frac1{(p-1)^2}-\cdots +(-1)^{k-1} \frac1{(p-1)^{k-1}}\right).
\end{equation}

By induction on $k$. If $k=1$, then $\varphi_1(n,m)= \varphi(n)$, by its definition \eqref{def_var_m_n}. Let $k\ge 2$.
Assume that \eqref{form_varphi_k_n_m} holds for $k-1$ and prove it for $k$. We have, by using Lemma
\ref{Lemma_recursion_varphi},
\begin{equation*}
\varphi_k(n,m)= \varphi(n) \sum_{d\mid m} \frac{\mu(d)}{\varphi(d)} \varphi_{k-1}(n,d)
\end{equation*}
\begin{equation*}
= \varphi(n) \sum_{d\mid m} \frac{\mu(d)}{\varphi(d)} \varphi(n)^{k-1} \prod_{p\mid d} \left(1- \frac1{p-1}+\frac1{(p-1)^2}-\cdots +(-1)^{k-2}
\frac1{(p-1)^{k-2}}\right)
\end{equation*}
\begin{equation*}
= \varphi(n)^k \prod_{p\mid m}  \left(1+ \frac{\mu(p)}{\varphi(p)}\left(1 - \frac1{p-1}+\frac1{(p-1)^2}-\cdots +(-1)^{k-2}
\frac1{(p-1)^{k-2}}\right)\right)
\end{equation*}
\begin{equation*}
= \varphi(n)^k \prod_{p\mid m}  \left( 1- \frac1{p-1}+\frac1{(p-1)^2}-\cdots +(-1)^{k-1} \frac1{(p-1)^{k-1}}\right),
\end{equation*}
which proves formula \eqref{form_varphi_k_n_m}. Now choosing $m=n$, \eqref{form_varphi_k_n_m} gives identity \eqref{gen_1}, which can be rewritten as
\eqref{gen_2}.
\end{proof}

\begin{proof}[Proof of Theorem {\rm \ref{Theorem_asympt_varphi_k}}]
Let $\varphi_k= \id_k * g_k$, that is, $g_k=\varphi_k * \mu \id_k$. Here the function $g_k(n)$ is multiplicative and
for any prime power $p^\nu$ ($\nu \ge 1$),
\begin{equation*}
g_k(p^{\nu}) =\varphi_k(p^{\nu})- p^k\varphi_k(p^{\nu-1}).
\end{equation*}

We obtain from \eqref{gen_2} that for $\nu\ge 2$,
\begin{equation} \label{g_k_p_nu}
g_k(p^{\nu}) = p^{(\nu-1) k} \left(1-\frac1{p} \right) \left((p-1)^k-(-1)^k \right) - p^k p^{(\nu-2)k} \left(1-\frac1{p} \right)
\left((p-1)^k-(-1)^k \right)=0,
\end{equation}
and for $\nu =1$,
\begin{equation} \label{g_k_p}
g_k(p) =  \varphi_k(p) -p^k = \left(1-\frac1{p} \right) \left((p-1)^k-(-1)^k \right) - p^k
\end{equation}
\begin{equation*} 
=  -(k+1)p^{k-1} +\cdots +(-1)^k k,
\end{equation*}
a polynomial in $p$ of degree $k-1$, with leading coefficient $-(k+1)$. Actually, we have 
\begin{equation} \label{g_k_ineq}
-(k+1)p^{k-1} < g_k(p) < 0
\end{equation}
for every integer $k\ge 2$ and every prime $p\ge 2$. To see this, note that by Lagrange's mean value theorem,
\begin{equation*}
k(p-1)^{k-1}< p^k-(p-1)^k < kp^{k-1},
\end{equation*}
and from \eqref{g_k_p} we deduce that 
\begin{equation*} 
g_k(p) > \left(1-\frac1{p} \right) \left(p^k -kp^{k-1}- (-1)^k \right) - p^k
\end{equation*}
\begin{equation*} 
= - (k+1) p^{k-1} + kp^{k-2}- (-1)^k \left(1-\frac1{p} \right)> -(k+1)p^{k-1}.
\end{equation*}

On the other hand, $p^k-(p-1)^k>1$, $(p-1)^k-(-1)^k<p^k$ imply that
\begin{equation*} 
g_k(p) <  \left(1-\frac1{p}\right) p^k  -p^k <0.
\end{equation*}

According to \eqref{g_k_ineq}, $|g_k(p)|< (k+1)p^{k-1}$ holds true for every $k\ge 2$ and every $p\ge 2$, and by \eqref{g_k_p_nu} we deduce that
\begin{equation*}
|g_k(n)| \le (k+1)^{\omega(n)} n^{k-1} \quad (n\in \N).
\end{equation*}

To obtain the desired asymptotic formula we apply elementary arguments. We have
\begin{equation*}
\sum_{n\le x} \varphi_k(n)= \sum_{d\le x} g_k(d) \sum_{\delta \le x/d} \delta^k
\end{equation*}
\begin{equation*}
=\sum_{d\le x} g_k(d) \left( \frac1{k+1} \left(\frac{x}{d}\right)^{k+1}  +O \left(\left(\frac{x}{d}\right)^k\right)  \right)
\end{equation*}
\begin{equation*}
= \frac{x^{k+1}}{k+1}  \sum_{d=1}^{\infty} \frac{g_k(d)}{d^{k+1}} + O \left( x^{k+1} \sum_{d>x} \frac{|g_k(d)|}{d^{k+1}} \right) +
O\left(x^k \sum_{d\le x} \frac{|g_k(d)|}{d^k}\right).
\end{equation*}

Here the main term is $\frac{C_k}{k+1} x^{k+1}$ by using the Euler product formula. To evaluate the error terms consider 
the Piltz divisor function $\tau_{k+1}(n)$, representing the number of ordered $(k+1)$-tuples $(a_1,\ldots,a_{k+1})\in \N^{k+1}$ such that
$a_1\cdots a_{k+1}=n$. We have $\tau_{k+1}(p^\nu)\ge \tau_{k+1}(p)=k+1$ for every prime power $p^{\nu}$ ($\nu \ge 1$), and
$\tau_{k+1}(n)\ge (k+1)^{\omega(n)}$ for every $n\in \N$.

We obtain
\begin{equation*}
\sum_{d>x} \frac{|g_k(d)|}{d^{k+1}}\le \sum_{d>x} \frac{(k+1)^{\omega(d)}}{d^2} \le \sum_{d>x} \frac{\tau_{k+1}(d)}{d^2} \ll \frac{(\log x)^k}{x},
\end{equation*}
and
\begin{equation*}
\sum_{d\le x} \frac{|g_k(d)|}{d^k} \le \sum_{d\le x} \frac{(k+1)^{\omega(d)}}{d} \le \sum_{d\le x} \frac{\tau_{k+1}(d)}{d} \ll
(\log x)^{k+1}, 
\end{equation*}
by using known elementary estimates on the Piltz divisor function. See, e.g., \cite[Lemma 3]{Tot2002}. This completes the proof.
\end{proof}

\section{Proof of Theorem \ref{Theorem_Menon_gen}} \label{Section_Proofs_Menon}

Let $M_k(n)$ denote the sum on the left hand side of \eqref{Menon_gen_f}. We have by the convolutional identity $f=(\mu*f)*\1$,
\begin{equation*}
M_k(n)= \sum_{\substack{a_1,\ldots, a_k=1 \\ (a_1\cdots a_k,n)=1 \\ (a_1+\cdots+a_k,n)=1}}^n \sum_{d\mid (a_1+\cdots +a_k-1,n)} (\mu*f)(d)=
\sum_{d\mid n} (\mu*f)(d) \sum_{\substack{a_1,\ldots, a_k=1 \\ (a_1\cdots a_k,n)=1 \\ (a_1+\cdots+a_k,n)=1\\ a_1+\cdots +a_k\equiv 1
\text{ (mod $d$)} }}^n 1
\end{equation*}
\begin{equation*}
= \sum_{d\mid n} (\mu*f)(d) \sum_{\substack{a_1,\ldots, a_k=1 \\ (a_1\cdots a_k,n)=1 \\  a_1+\cdots +a_k\equiv 1
\text{ (mod $d$)} }}^n \sum_{\delta \mid (a_1+\cdots+a_k,n)} \mu(\delta),
\end{equation*}
that is
\begin{equation} \label{first_M_k(n)}
M_k(n)= \sum_{d\mid n} (\mu*f)(d) \sum_{\delta \mid n} \mu(\delta) N_k(n,d,\delta),
\end{equation}
where
\begin{equation*}
N_k(n,d,\delta):= \sum_{\substack{a_1,\ldots, a_k=1 \\ (a_1\cdots a_k,n)=1 \\ a_1+\cdots+a_k\equiv 1 \text{ (mod $d$)} \\
a_1+\cdots +a_k\equiv 0 \text{ (mod $\delta$)} }}^n 1.
\end{equation*}

Next we evaluate the sum $N_k(n,d,\delta)$, where $d\mid n$, $\delta \mid n$ are fixed. If $(d,\delta)>1$, then $N_k(n,d,\delta)=0$,
the empty sum. So, assume that $(d,\delta)=1$. If $k=1$, then by using Lemma \ref{Lemma_cong_gen} we deduce
\begin{equation} \label{N_1}
N_1(n,d,\delta):= \sum_{\substack{a_1=1 \\ (a_1,n)=1 \\ a_1\equiv 1 \text{ (mod $d$)} \\
a_1 \equiv 0 \text{ (mod $\delta$)} }}^n 1=
\begin{cases} \frac{\varphi(n)}{\varphi(d)}, & \text{ if $\delta=1$},\\
0, & \text{ otherwise},
\end{cases}
\end{equation}
since for each term of the sum $\delta \mid a_1$ and $\delta \mid n$, which gives $\delta \mid (a_1,n)=1$, so $\delta=1$.

\begin{lemma} {\rm (recursion formula for $N_k(n,d,\delta)$)} \label{Lemma_recursion_N_k} Let $k\ge 2$, $d\mid n$, $\delta \mid n$, $(d,\delta)=1$. Then
\begin{equation} \label{recurs_N_k}
N_k(n,d,\delta) = \frac{\varphi(n)}{\varphi(d)\varphi(\delta)}\sum_{j\mid d} \mu(j) \sum_{t\mid \delta} \mu(t) N_{k-1}(n,j,t).
\end{equation}
\end{lemma}

\begin{proof}[Proof of Lemma  {\rm \ref{Lemma_recursion_N_k}}]
We have
\begin{equation*}
N_k(n,d,\delta)= \sum_{\substack{a_1,\ldots, a_{k-1}=1 \\ (a_1\cdots a_{k-1},n)=1}}^n \sum_{\substack{a_k=1\\ (a_k,n)=1 \\ a_k\equiv 1-a_1-\cdots-a_{k-1}\text{ (mod $d$)} \\ a_k\equiv -a_1-\cdots -a_{k-1} \text{ (mod $\delta$)} }}^n 1.
\end{equation*}

Using that $(d,\delta)=1$ and applying Lemma \ref{Lemma_cong_gen} we deduce that
\begin{equation*}
N_k(n,d,\delta)= \sum_{\substack{a_1,\ldots, a_{k-1}=1 \\ (a_1\cdots a_{k-1},n)=1\\ (a_1+\cdots +a_{k-1}-1,d)=1\\ (a_1+\cdots +a_{k-1},\delta)=1 }}^n \frac{\varphi(n)}{\varphi(d)\varphi(\delta)}
\end{equation*}
\begin{equation*}
= \frac{\varphi(n)}{\varphi(d)\varphi(\delta)} \sum_{\substack{a_1,\ldots, a_{k-1}=1 \\ (a_1\cdots a_{k-1},n)=1}}^n \sum_{j\mid (a_1+\cdots +a_{k-1}-1,d)} \mu(j) \sum_{t\mid (a_1+\cdots +a_{k-1},\delta)} \mu(t)
\end{equation*}
\begin{equation*}
= \frac{\varphi(n)}{\varphi(d)\varphi(\delta)} \sum_{j\mid d} \mu(j) \sum_{t\mid \delta} \mu(t) \sum_{\substack{a_1,\ldots, a_{k-1}=1 \\ (a_1\cdots a_{k-1},n)=1 \\ a_1+\cdots +a_{k-1}\equiv 1 \text{ (mod $j$)} \\  a_1+\cdots +a_{k-1}\equiv 0 \text{ (mod $t$)} }}^n 1
\end{equation*}
\begin{equation*}
= \frac{\varphi(n)}{\varphi(d)\varphi(\delta)} \sum_{j\mid d} \mu(j) \sum_{t\mid \delta} \mu(t) N_{k-1}(n,j,t).
\end{equation*}
\end{proof}

\begin{lemma}  \label{Lemma_evaluation_N_k} Let $k\ge 2$, $d\mid n$, $\delta \mid n$, $(d,\delta)=1$. Then
\begin{equation} \label{eval_N_k}
N_k(n,d,\delta) = \frac{\varphi(n)^k}{\varphi(d)\varphi(\delta)} \prod_{p\mid d} \left(1- \frac1{p-1}+\frac1{(p-1)^2}-\cdots +(-1)^{k-1} \frac1{(p-1)^{k-1}}\right)
\end{equation}
\begin{equation*}
\times \prod_{p\mid \delta} \left(1- \frac1{p-1}+\frac1{(p-1)^2}-\cdots +(-1)^{k-2} \frac1{(p-1)^{k-2}}\right).
\end{equation*}
\end{lemma}

\begin{proof}[Proof of Lemma  {\rm \ref{Lemma_evaluation_N_k}}]
By induction on $k$. If $k=2$, then by the recursion \eqref{recurs_N_k} and \eqref{N_1},
\begin{equation*}
N_2(n,d,\delta) = \frac{\varphi(n)}{\varphi(d)\varphi(\delta)}\sum_{j\mid d} \mu(j) \sum_{t\mid \delta} \mu(t) N_1(n,j,t)
\end{equation*}
\begin{equation}
= \frac{\varphi(n)}{\varphi(d)\varphi(\delta)}\sum_{j\mid d} \mu(j) \sum_{\substack{t\mid \delta\\t=1}} \mu(t) \frac{\varphi(n)}{\varphi(j)}
\end{equation}
\begin{equation}
= \frac{\varphi(n)^2}{\varphi(d)\varphi(\delta)} \sum_{j\mid d} \frac{\mu(j)}{\varphi(j)}= \frac{\varphi(n)^2}{\varphi(d)\varphi(\delta)} \prod_{p\mid d} \left(1-\frac1{p-1} \right).
\end{equation}

Hence, the formula is true for $k=2$. Assume it holds for $k-1$, where $k\ge 3$. Then we have, by the recursion \eqref{recurs_N_k},
\begin{equation*}
N_k(n,d,\delta) = \frac{\varphi(n)}{\varphi(d)\varphi(\delta)} \sum_{j\mid d} \mu(j) \sum_{\substack{t\mid \delta\\ (t,j)=1}}
\mu(t) \frac{\varphi(n)^{k-1}}{\varphi(j)\varphi(t)}
\end{equation*}
\begin{equation*}
\times  \prod_{p\mid j}  \left(1- \frac1{p-1}+\frac1{(p-1)^2}-\cdots +(-1)^{k-2} \frac1{(p-1)^{k-2}}\right)
\end{equation*}
\begin{equation*}
\times \prod_{p\mid t} \left(1- \frac1{p-1}+\frac1{(p-1)^2}-\cdots +(-1)^{k-3} \frac1{(p-1)^{k-3}}\right),
\end{equation*}
where the condition $(t,j)=1$ can be omitted, since $j\mid d$, $t\mid \delta$ and $(d,\delta)=1$. We deduce that
\begin{equation*}
N_k(n,d,\delta)= \frac{\varphi(n)^k}{\varphi(d)\varphi(\delta)} \sum_{j\mid d} \frac{\mu(j)}{\varphi(j)}
\prod_{p\mid j}  \left(1- \frac1{p-1}+\frac1{(p-1)^2}-\cdots +(-1)^{k-2} \frac1{(p-1)^{k-2}}\right)
\end{equation*}
\begin{equation*}
\times \sum_{t\mid \delta} \frac{\mu(t)}{\varphi(t)} \prod_{p\mid t} \left(1- \frac1{p-1}+\frac1{(p-1)^2}-\cdots +(-1)^{k-3} \frac1{(p-1)^{k-3}}\right)
\end{equation*}
\begin{equation*}
= \frac{\varphi(n)^k}{\varphi(d)\varphi(\delta)} \prod_{p\mid d}
\left(1- \frac1{p-1}\left(1-\frac1{p-1}+\frac1{(p-1)^2}-\cdots +(-1)^{k-2} \frac1{(p-1)^{k-2}}\right)\right)
\end{equation*}
\begin{equation*}
\times  \prod_{p\mid \delta} \left(1- \frac1{p-1}\left(\frac1{p-1}+\frac1{(p-1)^2}-\cdots +(-1)^{k-3} \frac1{(p-1)^{k-3}}\right)\right),
\end{equation*}
giving \eqref{eval_N_k}, which completes the proof of Lemma \ref{Lemma_evaluation_N_k}.
\end{proof}

Now we continue the evaluation of $M_k(n)$. According to \eqref{first_M_k(n)} and Lemma \ref{Lemma_evaluation_N_k}, we have
\begin{equation*}
M_k(n)=\sum_{d\mid n}(\mu*f)(d) \sum_{\substack{\delta\mid n\\ (\delta,d)=1}}\mu(\delta)
\frac{\varphi(n)^k}{\varphi(d)\varphi(\delta)} \prod_{p\mid d} \left(1- \frac1{p-1}+\frac1{(p-1)^2}-\cdots +(-1)^{k-1} \frac1{(p-1)^{k-1}}\right)
\end{equation*}
\begin{equation*}
\times \prod_{p\mid \delta} \left(1- \frac1{p-1}+\frac1{(p-1)^2}-\cdots +(-1)^{k-2} \frac1{(p-1)^{k-2}}\right)
\end{equation*}
\begin{equation*}
=\varphi(n)^k \sum_{d\mid n} \frac{(\mu*f)(d)}{\varphi(d)}
\prod_{p\mid d} \left(1- \frac1{p-1}+\frac1{(p-1)^2}-\cdots +(-1)^{k-1} \frac1{(p-1)^{k-1}}\right)
\end{equation*}
\begin{equation*}
\times \sum_{\substack{\delta\mid n\\ (\delta,d)=1}} \frac{\mu(\delta)}{\varphi(\delta)} \prod_{p\mid \delta} \left(1- \frac1{p-1}+\frac1{(p-1)^2}-\cdots +(-1)^{k-2} \frac1{(p-1)^{k-2}}\right),
\end{equation*}
where the inner sum is
\begin{equation*}
\prod_{\substack{p\mid n\\ p\nmid d}} \left(1+\frac{\mu(p)}{\varphi(p)}\left(1- \frac1{p-1}+\frac1{(p-1)^2}-\cdots +(-1)^{k-2} \frac1{(p-1)^{k-2}}\right)\right)
\end{equation*}
\begin{equation*}
= \prod_{p\mid n} \left(1-\frac1{p-1}+\frac1{(p-1)^2}-\cdots +(-1)^{k-1} \frac1{(p-1)^{k-1}}\right)
\end{equation*}
\begin{equation*}
\times \prod_{p\mid d} \left(1-\frac1{p-1}+\frac1{(p-1)^2}-\cdots +(-1)^{k-1} \frac1{(p-1)^{k-1}}\right)^{-1}.
\end{equation*}

This leads  to
\begin{equation*}
M_k(n)=\varphi(n)^k \prod_{p\mid n} \left(1-\frac1{p-1}+\frac1{(p-1)^2}-\cdots +(-1)^{k-1} \frac1{(p-1)^{k-1}}\right)  \sum_{d\mid n} \frac{(\mu*f)(d)}{\varphi(d)}
\end{equation*}
\begin{equation*}
=\varphi_k(n) \sum_{d\mid n} \frac{(\mu*f)(d)}{\varphi(d)},
\end{equation*}
by using \eqref{gen_1}, finishing the proof of Theorem \ref{Theorem_Menon_gen}.

\section{Acknowledgment} The author thanks the referee for useful remarks.

\end{document}